\documentclass[11pt]{amsart}
\theoremstyle{definition}

\newtheorem{theorem}{Theorem}[section]

\newtheorem*{acknowledgement}{Acknowledgments}

\newtheorem{lemma}[theorem]{Lemma}

\newtheorem{proposition}[theorem]{Proposition}

\usepackage{url,upref,amscd,amsmath,amsthm}
\usepackage{times}
\usepackage{fancyhdr}
\usepackage[psamsfonts]{amssymb}
\usepackage{cite}
\title{The automorphism group of the group of unitriangular matrices over a field}
\author{Ayan Mahalanobis}
\address{Indian Institute of Science Education and Research Pune, Sai
  Trinity Building, Pashan, Pune 411021, INDIA}
\email{ayanm@iiserpune.ac.in}
\thanks{Research supported by a NBHM grant}
\begin{document}
\begin{abstract}
This paper finds the generators of the automorphism group of the group
of unitriangular matrices over a field. Most of this paper is an
exposition of the work of V.M. Lev\u{c}huk, part of which is in
Russian. Some proofs are of my own.
\end{abstract}
\maketitle
\section{Introduction}
The automorphism group of the group of unitriangular matrices over a
field was studied by many~\cite{weir,levchuk2,maginnis}. In this
direction, the first paper was in Russian, published by Pavlov in
1953. Pavlov studies the automorphism group of unitriangular matrices
over a finite field of
odd prime order. Weir~\cite{weir} describes the automorphism group of
the group of unitriangular matrices over an finite field of odd
characteristic. Maginnis~\cite{maginnis} describes it for the field of
two elements and finally Lev\u{c}huk~\cite{levchuk2} describes the
automorphism group of the group of unitriangular matrices over an arbitrary ring. 
 In this
expository article, we will study the automorphism group of the group of unitriangular matrices over an arbitrary field $\mathbf{F}$.

To start, we denote the algebra of all lower niltraingular matrices over $\mathbf{F}$, of
size $d$, by NT$(d,\mathbf{F})$. This is the set of all matrices that
have zero on and above the diagonal, and arbitrary field element
(possible non-zero) below the main diagonal. It is known to be a
\emph{nilpotent algebra}, $M^d=0$, for all
$M\in\text{NT}(d,\mathbf{F})$. Where $0$ is the zero matrix of size
$d$. The general method, that we describe below, works only when $d$ is
greater than 4. The case of $d=3$ and $d=4$ can be computed by hand
and was done by Lev\u{c}huk~\cite{levchuk2}. Henceforth, we assume
that $d\geq 5$.

One can define two operations on the set NT$(d,\mathbf{F})$. 
\begin{description}
\item The first operation is $\centerdot$, defined as $a\centerdot b=a+b+ab$.
\item The second operation is $\ast$, defined as $a\ast b=ab-ba$.
\end{description}
It is known that $\left(\text{NT}(d,\mathbf{F}), \centerdot\right)$ is isomorphic to UT$(d,\mathbf{F})$, the group of (lower) \emph{unitriangular matrices}. The isomorphism being $x\mapsto 1+x$, where 1 is the identity matrix of size $d$. This groups is also known as the \textit{associated group} of the ring NT$(d,\mathbf{F})$. In this paper, we will denote the associated group of NT$(d,\mathbf{F})$ by UT$(d,\mathbf{F})$.

The second operation is a Lie bracket, it is known that
$\left(\text{NT}(d,\mathbf{F}),+,\ast\right)$ is a Lie algebra. In
this paper, we will denote this Lie algebra by
NT$^\ast(d,\mathbf{F})$. It is not hard to see, in the light of
Equations~\ref{grpeqn2} and~\ref{lieeqn} later, that this Lie algebra is the
same as the graded Lie Algebra of the group UT$(d,\mathbf{F})$.

For $i>j$ and $x\in\mathbf{F}$, we define the matrix unit $xe_{i,j}$ to be the $d\times d$
matrix with $x$ in the $(i,j)$ position and $0$ everywhere else.

The \emph{defining relations} in these three algebraic objects are the relations in
the field $\mathbf{F}$ and the following:
\paragraph{The algebra NT$(d,\mathbf{F})$}
\begin{equation}\label{ringeq}
(xe_{i,j})(ye_{k,l})=\left\{
\begin{array}{ccc}
xye_{i,l} &\text{whenever}&j=k\\
0&\text{otherwise}
\end{array}
\right.
\end{equation}
\paragraph{The group UT$(d,\mathbf{F})$}
\begin{eqnarray}
(xe_{i,j})\centerdot(ye_{i,j})=(x+y)e_{i,j}\label{grpeqn1}\\
{[}xe_{i,j},ye_{k,l}{]}=\left\{
\begin{array}{ccr}
xye_{i,l}&\text{whenever}&j=k\\
-xye_{k,j}&\text{whenever}&i=l\\
0&\text{otherwise}
\end{array}\right.\label{grpeqn2}
\end{eqnarray}
\paragraph{The Lie algebra NT$^\ast(d,\mathbf{F})$}
\begin{eqnarray}\label{lieeqn}
(xe_{i,j})\ast(ye_{k,l})=\left\{
\begin{array}{ccr}
\;\; xye_{i,l}&\text{whenever}&j=k\\
-xye_{k,j}&\text{whenever}&i=l\\
0&\text{otherwise}
\end{array}\right.
\end{eqnarray}

From the Relation~\ref{grpeqn2}, it follows, that a set of
generators for the UT$(d,\mathbf{F})$, is of the form $xe_{i+1,i}$,
$x\in\mathbf{F}$ and $i=1,2,\ldots,d-1$. This is actually a set of
\emph{minimal} generators. Since the commutator relation and the
relation in the Lie algebra are the same, the same set
acts as generators for the Lie algebra as well.

Define
\[\Gamma_k=\left\{M=\sum m_{i,j}e_{i,j}\in\text{UT}(d,\mathbf{F});\;\;
  m_{i,j}=0,\;\;i-j<k\right\},\] in other words, the
$\Gamma_1=\text{UT}(d,\mathbf{F})$. The subgroup $\Gamma_2$ is the commutator of
UT$(d,\mathbf{F})$. It consists of all lower niltriangular matrices
with the first subdiagonal entries zero. The first subdiagonal can be
specified by all entries $(i,j)$ with $i-j=1$. Similarly $\Gamma_2$
consists of all matrices with the first two subdiagonals zero and so
on. It follows that $\Gamma_d=0$.

It follows from Relation~\ref{grpeqn2}, if $i-j=k_1$ and $k-l=k_2$
and $[xe_{i,j},ye_{k,l}]$ is non-zero, then the commutator is
$xye_{i,l}$ or $xye_{k,j}$. In both these cases, $i-l$ or $k-j$ equals
$k_1+k_2$. Taking these into account, one can prove the next
proposition.
\begin{proposition}
In UT$(d,\mathbf{F})$, the lower central series and the upper central
series are identical and is of the form $\text{UT}(d,\mathbf{F})=\Gamma_1>\Gamma_2>\ldots >\Gamma_{d-1}>\Gamma_d=0$.
\end{proposition}
There is an interesting and useful connection between the normal subgroups of
UT$(d,\mathbf{F})$ and ideals of NT$^\ast(d,\mathbf{F})$. The connection can be motivated by a simple observation: Let $1+L$ is in UT$(d,\mathbf{F})$, i.e., $L\in\text{NT}(d,\mathbf{F})$.
Then \begin{equation}\left(1+e_{ij}\right)^{-1}(1+L)\left(1+e_{ij}\right)=1+L+\left(L\ast e_{ij}\right),\end{equation} which implies that under suitable conditions, elements in a normal subgroup of UT$(d,\mathbf{F})$ are closed under Lie bracket. Conversely, under suitable condition, an ideal of NT$^\ast(d,\mathbf{F})$ is a  normal subgroup of UT$(d,\mathbf{F})$.

Furthermore one should also notice, if a subgroup $H$ of UT$(d,\mathbf{F})$ is abelian, then we have $(1+L)(1+M)=(1+M)(1+L)$ for $1+L,1+M\in H$; which implies that $L\ast M=0$, i.e., if a subgroup is abelian and an ideal, then that ideal is abelian as well and vice versa.

Notice that, in the motivation above, we have represented an element of the group UT$(d,\mathbf{F})$ as $1+L$, where $L\in\text{NT}(d,\mathbf{F})$. This is not necessary, we can use $L$ and the operation $\centerdot$. However, since $1+L$ makes the group operation  matrix multiplication, this makes our motivation transparent. From now on, elements in UT$(d,\mathbf{F})$ will be represented as elements of NT$(d,\mathbf{F})$, with the operation $\centerdot$.   

For $i>j$, let us define $\mathbf{N}_{i,j}$ to be the subset of NT$(d,\mathbf{F})$ with all rows less then the $i\textsuperscript{th}$ row zero and all columns greater than the $j\textsuperscript{th}$ column zero. It is a rectangle and Weir~\cite{weir} calls it a \emph{partition} subgroup. It is straightforward to see that $\mathbf{N}_{i,j}$ is an abelian (two-sided) ideal of the ring NT$(d,\mathbf{F})$. From this it follows that $\mathbf{N}_{i,j}$ is an abelian ideal of NT$^\ast(d,\mathbf{F})$.
\begin{lemma}\label{lem1}
If $H$ is a maximal abelian normal subgroup of UT$(d,\mathbf{F})$ or a
maximal abelian ideal of NT$^\ast(d,\mathbf{F})$, then 
\begin{itemize}
\item $H^2\subseteq H$.
\item $H^2\subseteq\mathbf{N}_{d,1}$.
\item $\alpha\gamma\beta+\beta\gamma\alpha=0$ for $\alpha,\beta\in H$
  and $\gamma\in\text{NT}(d,\mathbf{F})$.
\end{itemize}
\end{lemma}
\begin{proof}
From maximality, it follows that $H$ contains the annihilator of the
ring NT$(d,\mathbf{F})$, i.e., the subset
$\left\{x\;|\;xy=0=yx,\;\text{for all}\;
  y\in\text{NT}(d,\mathbf{F})\right\}$. We show that $H^2$ is
contained in the annihilator. We only work with the associated group,
the proof for Lie algebra is identical.

Since $H$ is normal, for any $\alpha,\beta\in H$ and $\gamma=xe_{ij}\in\text{NT}(d,\mathbf{F})$, $\alpha$ commutes with $(-xe_{ij})\centerdot \beta\centerdot(xe_{ij})$. This implies that
\[\alpha\centerdot\left(\beta+\beta(xe_{ij})-(xe_{ij})\beta\right)=\left(\beta+\beta(xe_{ij})-(xe_{ij})\beta\right)\centerdot\alpha\]
Since $H$ is abelian, $\alpha\beta=\beta\alpha$, 
\[\left(\alpha\beta(xe_{ij})+(xe_{ij})\beta\alpha\right)-\left(\alpha(xe_{ij})\beta+\beta(xe_{ij})\alpha
\right)=0\]
Now notice that the matrix represented in the first parenthesis has non-zero elements only on the $i^\textsuperscript{th}$ row and the $j^\textsuperscript{th}$ column. On the other hand the matrix represented by the second parenthesis has both the $i^\textsuperscript{th}$ row and the $j^\textsuperscript{th}$ column zero. Hence the equality is possible only when both the matrices are zero.

This implies that $\alpha\beta(xe_{ij})=0=(xe_{ij})\beta\alpha$ for
$i>j$ and $i=2,3,\ldots,d$ and $j=1,2,\ldots,d-1$. It is also clear
that $\alpha\gamma\beta+\beta\gamma\alpha=0$ for $\gamma=xe_{i,j}$.

Notice that for any
$d\times d$ matrix $A$, $Ae_{i,j}$ is the matrix with only the
$j\textsuperscript{th}$ column non-zero and the contents are the
contents of the $i\textsuperscript{th}$ column. Thus
$\alpha\beta(xe_{ij})=0$ implies that the $i\textsuperscript{th}$
column of $\alpha\beta$ is zero for
$i=2,3,\ldots,d$. Similarly one can show, that the
$j\textsuperscript{th}$ row is zero for
$j=1,2,\ldots,d-1$. Then it follows that $H^2\subseteq\mathbf{N}_{d,1}$.

Since any matrix in NT$(d,\mathbf{F})$ can be written as a linear
combination of elementary matrices $xe_{i,j}$, the proof that $H^2$ is contained
in the annihilator is complete. Furthermore, since $H$ is closed under
addition, we have $\alpha\gamma\beta+\beta\gamma\alpha=0$, for all $\gamma\in\text{NT}(d,\mathbf{F})$.
\end{proof}
\begin{theorem}[Lev\u{c}huk, 1976]
A maximal abelian normal subgroups of UT$(d,\mathbf{F})$ is also a
maximal abelian ideals of NT$^\ast(d,\mathbf{F})$ and vice versa.
\end{theorem}
\begin{proof}
Let $N$ be a maximal abelian normal subgroup of UT$(d,\mathbf{F})$. Then construct the subgroup 
\[N^\prime=\left\langle N\cup \{\alpha-\beta\},\;\;\alpha,\beta\in N\right\rangle\]
Clearly $N^\prime$ is an abelian subgroup. Since matrix multiplication distributes over addition, the subgroup $N^\prime$ is normal. Since $N$ is contained in $N^\prime$, the maximality implies that $N=N^\prime$ and so $N$ contains the sum of any two of its elements.

To show that it is closed under Lie bracket, notice that, $N$ is
normal implies $(-xe_{ij})\centerdot\alpha\centerdot(xe_{ij})$ is in $N$. This implies that $\alpha+\left(\alpha\ast xe_{ij}\right)$ is in $N$. Since $N$ is closed under addition, $\alpha\ast xe_{ij}$ is in $N$.

The fact that $N$ is an abelian ideal follows from the fact that $N$ is an abelian group.

Conversely, assume that $I$ is a maximal abelian ideal of the Lie ring
NT$^\ast(d,\mathbf{F})$. Then from Lemma \ref{lem1}, $I^2\subseteq I$. This proves that $I$ is a subring, this implies that it is closed under $\centerdot$. So $I$ is a subgroup, and since $\alpha\ast xe_{ij}\in I$, for $\alpha\in I$ and $xe_{ij}\in\text{NT}(d,\mathbf{F})$, $\alpha(xe_{ij})-(xe_{ij})\alpha\in I$. This shows that $I$ is a normal subgroup.  
\end{proof}
\section{Maximal abelian ideals of NT$^\ast(d,\mathbf{F})$}
Notice that the centralizer of any set in the ring
NT$(d,\mathbf{F})$, the associated group UT$(d,\mathbf{F})$, and the Lie ring  NT$^\ast(d,\mathbf{F})$ is
identical. Let us look at the centralizer of
$\mathbf{N}_{i,j}$. Notice that, if
$\left(\sum\limits_{m>n}a_{m,n}e_{m,n}\right)N_{i,j}=N_{i,j}\left(\sum\limits_{m>n}a_{m,n}e_{m,n}\right)$. Then
the left hand side is a linear combination of the rows of $N_{i,j}$ and
the right hand side is a linear combination of columns of
$N_{i,j}$. Since the entries of $N_{i,j}$ are arbitrary field
elements, the only way that this is possible is that
$e_{m,n}\mathbf{N}_{i,j}=0$ and $\mathbf{N}_{i,j}e_{m,n}=0$ for $m>n$.
So to find the centralizer is to look for $m,n$ with $m>n$, such that
$e_{m,n}N_{i,j}=0=N_{i,j}e_{m,n}$. Now it is easy to see that the centralizer
\[\mathcal{C}\left(\mathbf{N}_{i,j}\right)=\mathbf{N}_{j+1,i-1}.\]
Since $\mathcal{C}\left(\mathbf{N}_{i+1,i}\right)=\mathbf{N}_{i+1,i}$,
if $\mathbf{N}_{i+1,i}$ is properly contained in an abelian ideal,
then that ideal is contained in the centralizer; which is
impossible. This proves that
\begin{equation}\label{mab1}
\mathbf{N}_{i+1,i}
\end{equation} 
is a maximal abelian ideal of
NT$^\ast(d,\mathbf{F})$ for $i=1,2,\ldots,d-1$.  Further notice
that, \[\Gamma_k=N_{k+1,1}+N_{k+2,2}+\ldots+N_{d,d-k},\] taking
intersection of partition subgroups, it is easy to
see that
$C\left(\Gamma_k\right)=N_{d-k+1,k}$. In particular, if $d$ is even,
i.e., $d=2k$ for some integer $k$, then
$C\left(\Gamma_k\right)=N_{k+1,k}$. 

\subsection{Are there any other
maximal abelian ideals of NT$^\ast(d,\mathbf{F})$?}

Let $H$ be a maximal abelian ideal of
NT$^\ast(d,\mathbf{F})$. Following Lev\u{c}huk~\cite{levchuk1}, we
define $H_{i,j}$ to be a subset of $\mathbf{F}$, whose elements are in
the $(i,j)$ position of a matrix belonging to $H$.

Let $m$ be the smallest, and $n$ be the largest integer such that
$H_{m,1}\neq 0$ and $H_{d,n}\neq 0$.  Since $H$ is an ideal, for
$i>j>1$ and $i<m$, $H\ast e_{j,1}\in H$. This implies that
$H_{i,j}e_{i,1}\in H$. Now for $i<m$, $H_{i,1}=0$, hence $H_{i,j}=0$
for $i<m$.

Similarly, one can show that $H_{i,j}=0$ for $j>n$ by looking at the
fact, $H\ast e_{d,j}\in H$.

Two things can happen, either $n<m$ or $n\geq m$. In the first case, it
is clear that $H$ is contained in $\mathbf{N}_{m,m-1}$ and is thus $\mathbf{N}_{m,m-1}$.

If we assume that $n\geq m$, then the description of $H$ is bit
involved. It gives rise to maximal abelian ideals of
\emph{exceptional} type.

First notice, if $n=m$, then $H_{m,j}=0$ for $j>1$ and
$H_{i,n}=0$ for $i<d$. This follows from the fact that
$H^2\subseteq\mathbf{N}_{d,1}$ (see Lemma \ref{lem1}). 

In this case ($n=m$), let $\alpha,\beta\in H$. Then
$\alpha\ast\beta=0$. If we write
$\alpha=\sum\limits_{i>j}\alpha_{i,j}e_{i,j}$ and
$\beta=\sum\limits_{i>j}\beta_{i,j}e_{i,j}$, then
$\alpha_{n,1}\beta_{d,n}-\beta_{n,1}\alpha_{d,n}=0$. Now notice that
$H$ being closed under addition, we can assume that $\alpha_{n,1}$ and
$\alpha_{d,n}$ are non zero. This implies that the maximal abelian
ideal $H$ is of the form: for a fixed $c\in\mathbf{F}$
\begin{equation}\label{mab2}
\left\{\mathbf{N}_{i+1,i-1}+xe_{i,1}+xce_{d,i};\;\; x\in\mathbf{F}\right\}.
\end{equation}

Now let us consider the case of $n>m$, in this case we first show that
$n>m+1$ is impossible. 

For $n>m$, we show that $H_{m,i}=0$ for $i>1$. Notice that for
$n>m>i>1$, $e_{n,m}\ast(H\ast e_{i,1})\in H$. Since $H\ast H=0$, this
implies that $H_{d,n}H_{m,i}e_{d,1}=0$. Since $H_{d,n}\neq 0$,
$H_{m,i}=0$ for $i>1$.

In a similar way, looking at $e_{d,j}\ast(H\ast e_{n,m})\in H$ shows
us that $H_{j,n}=0$ for $j<d$.

Then for $n>m+1$, $(H\ast e_{n,m+1})\ast e_{m+1,m}\in H$. This implies
that $H_{d,n}e_{d,n}\in H$. The fact that $H\ast H=0$, gives us that
$H_{d,n}H_{m,1}e_{d,1}=0$. However this is impossible. Hence $n>m+1$
is impossible.

Now we show that, if $n=m+1$, then $2\mathbf{F}=0$, i.e., $\mathbf{F}$
has characteristic $2$. Since $H_{m,1}$ and $H_{d,m+1}$ are both
non-zero. Since $H$ is closed under addition, there is a matrix $\alpha\in
H$, where $\alpha=\sum\limits_{i>j}\alpha_{i,j}e_{i,j}$ and
$\alpha_{m,1}\neq 0$ and $\alpha_{d,m+1}\neq 0$.

From Lemma \ref{lem1}, we know that $2\alpha(xe_{m+1,m})\alpha=0$ for
any $x\in\mathbf{F}$, which says that
$2x\left(\alpha_{d,m+1}\alpha_{m,1}e_{d,j}\right)=0$. Since $x$ is
arbitrary, we have $2\mathbf{F}=0$.

From Lemma \ref{lem1}, $H^2\subset\mathbf{N}_{d,1}$, this
  implies, $H_{m+1,i}=0$ for $i>1$ and $H_{j,m}=0$ for $j<d$. Let
  $\alpha=\sum\limits_{i>j}\alpha_{i,j}e_{i,j}$ is in $H$. Since $H$
  is closed under addition, we may assume that
  $\alpha_{m,1},\alpha_{m+1,1},\alpha_{d,m}$ and $\alpha_{d,m+1}$ are
  all nonzero. 
Then $(H\ast e_{m+1,m})\ast H=0$, implies that $(\alpha\ast
e_{m+1,m})\ast\beta=0$, where
$\beta=\sum\limits_{i>j}\beta_{i,j}e_{i,j}$. This is the same as
saying that
$\alpha_{d,m+1}\beta_{m,1}+\alpha_{m,1}\beta_{d,m+1}=0$. This implies
that there is a $c\in \mathbf{F}$ such that
$\alpha_{d,m+1}=c\alpha_{m,1}$.

Then the maximal abelian ideal $H$ is of the form: for a
$c\in\mathbf{F}$ and $i=2,3,\ldots,d-2$,
\begin{equation}\label{mab3}
\mathbf{N}_{i+2,i-1}+ae_{i+1,1}+be_{d,i}+xe_{i,1}+cxe_{d,i+1};\; a,b,x\in\mathbf{F}.
\end{equation}
 So by now we have proved a theorem.
\begin{theorem}[Weir, 1955; Levchuk 1976]
The maximal abelian ideals of the Lie ring NT$^\ast(d,\mathbf{F})$ are
of the following form: (\ref{mab1}), (\ref{mab2}) and (\ref{mab3}) above. The (\ref{mab3})
occurs only when the field is of characteristic 2. 
\end{theorem}
\section{The automorphism group of UT$(d,\mathbf{F})$}
In this section we describe all the automorphisms of the group UT$(d,\mathbf{F})$. The automorphisms are as follows:
\begin{description}
\item[Extremal automorphisms -- Aut$_E$] These automorphisms arise from
  the maximal abelian ideals of exceptional type. As we saw, the
  maximal abelian ideals of exceptional type are different for a field
  of characteristic 2. So, we will have two different types of
  automorphisms. One for even characteristic  and other for the field of
  odd characteristic.
\begin{description}
\item[Odd Characteristic]
\begin{equation}\label{aut1}
xe_{2,1}\mapsto xe_{2,1}+axe_{d,2}+x^\lambda e_{d,1}
\end{equation}
Where $x^\lambda:\mathbf{F}\rightarrow\mathbf{F}$ is a map that
satisfies the equation $(x+y)^\lambda-x^\lambda-y^\lambda=axy$ and
$a=2^\lambda-2(1^\lambda)$. All other generators remain fixed.

Similarly, one can define $xe_{d,d-1}=xe_{d,d-1}+axe_{d-1,1}+x^\lambda
e_{d,1}$. All other generators remain fixed and the $\lambda$
satisfies the above relations.

If $\mathbf{F}$ is of even characteristic, then
$a=0^\lambda$. It is easy to see that, since $0^\lambda=0$, $a=0$. So
in the case of the characteristic of the field to be even, the
extremal automorphisms become the central automorphisms.

Clearly automorphisms of this form generate a subgroup of the full
automorphism group of UT$(d,\mathbf{F})$, denoted by Aut$_E$ and is
isomorphic to $\mathbf{F}^+\oplus\mathbf{F}^+$.
\item[Even Characteristic]
\begin{equation}\label{aut2}
xe_{2,1}\mapsto xe_{2,1}+axe_{d,3}
\end{equation}
all other generators remain fixed.
Similarly one can define $xe_{d,d-1}\mapsto
xe_{d,d-1}+axe_{d-1,1}$. Again this automorphism group Aut$_E$ is
isomorphic to $\mathbf{F}^+\oplus\mathbf{F}^+$. We will later show
that these automorphisms are only possible when the field $\mathbf{F}$
is the field of two elements $\mathbb{Z}_2$.
\end{description}
\item[Flip automorphism -- Aut$_F$]
This automorphism is given by flipping the matrix by the anti-diagonal
and is given by the equation:
\begin{equation}\label{flip}
xe_{i,j}\mapsto xe_{d-j+1,d-i+1}
\end{equation}
This is clearly an automorphism of order 2 and forms a subgroup of the
automorphism group and will be denoted by Aut$_F$.
\item[Diagonal automorphisms -- Aut$_D$]
This automorphism is conjugation by a diagonal matrix. Diagonal
matrices are defined as matrices with only non-zero terms in the main
diagonal and everything else zero. Let $D=[x_1,x_2,\ldots,x_d]$ be a
diagonal matrix, with $x_1,x_2,\ldots,x_d$ as the non-zero diagonal
entries in the respective rows. Then
$D^{-1}xe_{i,j}D=d_i^{-1}xd_je_{i,j}$. So a diagonal matrix maps
$xe_{i,j}\mapsto d_i^{-1}xd_je_{i,j}$. The kernel of this map is the set of all scalar matrices, i.e.,
$x_1=x_2=\cdots=x_d$. This is clearly a subgroup of the automorphism
group, which is of the form
$\mathbf{F}^\times\times\mathbf{F}^\times\times\ldots\times\mathbf{F}^\times$
($d-1$ times), and will be denoted as Aut$_D$.
\item[Field automorphisms -- Aut$_A$] This automorphisms can be
  described  as
\begin{equation}
xe_{i+1,i}\mapsto x^\mu e_{i+1,i}\;\;i=1,2,\ldots,d-1.
\end{equation}
Where $\mu:\mathbf{F}\rightarrow\mathbf{F}$ is a field automorphism.
\item[Inner automorphisms -- Aut$_I$] This is the well known
 normal subgroup of the automorphism group in any non-abelian group; where $x\mapsto g^{-1}xg$ for
  some $g\in\text{UT}(d,\mathbf{F})$ and $x\in\text{UT}(d,\mathbf{F})$.
\item[Central automorphisms -- Aut$_C$] Central automorphisms are the
  centralizers of the group of inner automorphisms in the group of
  automorphisms. The simplest way to explain them is to ``multiply''
  the generators with an element of the center. In the case of
  UT$(d,\mathbf{F})$ it is
\begin{equation}
xe_{i+1,i}\mapsto xe_{i+1,i}+x^\lambda e_{d,1}
\end{equation}
Where $\lambda$ is a linear map of $\mathbf{F}^+$ to itself.
\end{description}
\subsection{Why are these the only automorphisms of
  UT$(d,\mathbf{F})$?}
We know that any automorphism $\phi$ of any group maps a maximal abelian normal
subgroup to a maximal abelian normal subgroup. Out first lemma uses that to prove:
\begin{lemma}
Let $\phi$ be an automorphism of UT$(d,\mathbf{F})$. Then
$\mathbf{N}_{i+1,i}^\phi$ is either $\mathbf{N}_{i+1,i}$ or
$\mathbf{N}_{d-i+1,d-i}$, for $i=1,2,\ldots,d-1$.
\end{lemma}
\begin{proof}
Notice that the centralizer $C(\Gamma_k)$ of $\Gamma_k$, the $k\textsuperscript{th}$
element in the central series is characteristic.
If $d$ is even, and $d=2k$ for some $k$, then
$C\left(\Gamma_k\right)=N_{k+1,k}$. Hence $N_{k+1,k}^\phi=N_{k+1,k}$.
 Further notice that for $i<\frac{d}{2}$,
$\mathbf{N}_{i+1,i}\cap\mathbf{N}_{d-i+1,d-i}=\mathbf{N}_{d-i+1,i}$
is a characteristic subgroup of UT$(d,\mathbf{F})$. Since flip is an
automorphism of UT$(d,\mathbf{F})$, any characteristic subgroup must
be symmetric about the second diagonal.

Then $\mathbf{N}_{d-i+1,i}$ is the maximal characteristic subgroup of
UT$(d,\mathbf{F})$ contained in
both $\mathbf{N}_{i+1,i}$ and $\mathbf{N}_{d-i+1,d-i}$. This means
that it must be contained as a maximal characteristic subgroup in
$\mathbf{N}_{i+1,i}^\phi$. So $\mathbf{N}_{i+1,i}^\phi$ has two choices,
$\mathbf{N}_{i+1,i}$ or $\mathbf{N}_{d-i+1,d-i}$.
\end{proof}
It is
important to notice here that $xe_{2,1}$ and
$xe_{d,d-1}$ are not only contained in the maximal abelian normal
subgroups $N_{2,1}$ and $N_{d,d-1}$ . They
are also contained in the exceptional subgroups. Let us call the maximal
abelian normal subgroup of type~(\ref{mab2}) (or of type~(\ref{mab3}),
when characteristic of the field is $2$) containing $xe_{2,1}$ as $\mathbf{A}_2$
and the maximal abelian normal subgroup containing
$xe_{d,d-1}$ as $\mathbf{A}_{d-1}$.

It is clear from the last lemma, that $\mathbf{A}_2^\phi$ is either
$\mathbf{A}_2$ or $\mathbf{A}_{d-1}$. Then clearly, if necessary, composing $\phi$ with the flip
automorphism, we claim that $\text{UT}(d,\mathbf{F})/\Gamma_2$ is
invariant under $\phi$. We can actually say more,
$xe_{i+1,i}^\phi=x^{\lambda_i} e_{i+1,i} \mod \Gamma_2$, for $i=1,2,\ldots,d-1$. Where
$\lambda_i:\mathbf{F}\rightarrow\mathbf{F}$ is a map. Now let us try to
understand the map $\lambda_i$. Since $\phi$ is an automorphism, each
map $\lambda_i$ is a bijection.

Now recall the relations in UT$(d,\mathbf{F})$ (Equations
\ref{grpeqn1}). It follows from the relation $xe_{i,j}\centerdot ye_{i,j}=(x+y)e_{i,j}$, if $xe_{i+1,i}\mapsto
x^{\lambda_i}e_{i+1,i}$ then $\lambda_i$ is a linear map of
$\mathbf{F}^+$.

Furthermore, since
$\left[xe_{i+1,i},ye_{i,i-1}\right]=\left[ye_{i+1,i},xe_{i,i-1}\right]$,
$x^{\lambda_i}y^{\lambda_{i-1}}=y^{\lambda_i}x^{\lambda_{i-1}}$, for
$i=2,3,\ldots,d-1$. Also, since
$\left[xe_{2,1},ye_{3,2}\right]^\phi=\left[ye_{2,1},xe_{3,2}\right]$,
  $x^{\lambda_1}y^{\lambda_2}=y^{\lambda_1}x^{\lambda_2}$.

Taking all these together, it follows that
$\lambda_1=k_1\lambda_2=k_3\lambda_3=\cdots=k_{d-2}\lambda_{d-1}$. Where
$k_i$ are nonzero fields elements.

So now we are in a position to claim, that composing $\phi$ with a
field automorphism and a diagonal automorphism, $\phi$ maps like
the identity UT$(d,\mathbf{F})/\Gamma_2$.

As we saw from the above lemma, $\phi$, (after composing with the flip,
if necessary) maps $A_2$ and $A_{d-1}$ to itself. So it follows that
the $xe_{2,1}^\phi$ and $xe_{d,d-1}^\phi$ are in $A_2$ and $A_{d-1}$
respectively. In case of odd characteristic, the description of the
extremal automorphism is obvious and is defined in Equation \ref{aut1}.

In case of the even characteristic we need to say more. Notice that in
the case that the characteristic of the field $\mathbf{F}$ being even,
the maximal abelian normal subgroup containing $e_{2,1}$ is
\[\mathbf{N}_{4,1}+ae_{3,1}+be_{d,2}+xe_{2,1}+cxe_{d,3}\]
where $a,b\in\mathbf{F}$, $c\in\mathbf{F}^\times$,
$x\in\mathbf{F}$. We want to know more about the automorphism that
moves $e_{2,1}$. Using the flip automorphism, if necessary, we can
assume that $A_2^\phi=A_2$. So the only choice is 
\[xe_{2,1}\mapsto xe_{2,1}+axe_{d,3}+x^\lambda e_{d,2}+ x^\mu e_{d,1}\]
where $\lambda,\mu :\mathbf{F}\rightarrow\mathbf{F}$.
From relation \ref{grpeqn1}, we see that $(x+y)^\lambda =x^\lambda+y^\lambda$. This implies that $0^\lambda=0$. Also it follows that
$(x+y)^\mu=x^\mu + y^\mu + x^\lambda y$. This implies that
$(x+y)^\mu=x^\mu+y^\mu+1^\lambda xy$. Putting $x=0$ and $y=1$, we have
that $0^\mu=0$. Putting $x=1$ and $y=1$ we get
$1^\lambda=0$. Since $x^\lambda=1^\lambda x$, this implies that
$x^\lambda=0$, i.e., $\lambda$ is the zero map.

Once $\lambda$ is the zero map, clearly the $e_{d,1}$ entry is not
necessary. Hence we have the Equation \ref{aut2}.

From the commutator relations (Relation \ref{grpeqn2}), we see that the
\[xye_{3,1}\mapsto xye_{3,1}+ax^2ye_{d,1}+axye_{d,2}.\] Then
interchanging $x$ and $y$, we see that $axy(y-x)=0$. Since $a\neq 0$,
so in the field $\mathbf{F}$, for any two distinct element, 
one is zero. This means that $\mathbf{F}=\mathbb{Z}_2$.

Notice that
for $i=s$, $[xe_{i+1,i},ye_{s,t}]=xye_{i+1,t}$. So if $s-t=k$, then
$(i+1)-l=k+1$. Using this idea one can clear each and every subdiagonals,
one after another, starting with $k=2$. This means that by suitable conjugation,
$xe_{i+1,i}^\phi$ will have no non-zero entries except the $(d,1)$
entry. In the case of $xe_{2,1}^\phi$ and $xe_{d,d-1}^\phi$ one can
actually clear all non-zero entries, including the $(d,1)$ entry, using conjugation.
One can choose the conjugators in such a way that this gives
rise to an inner automorphism. This proves the following lemma:
\begin{lemma}
Let $\Gamma_k$ be as defined. For an automorphism $\phi$ of
UT$(d,\mathbf{F})$, which fixes $\text{UT}(d,\mathbf{F}) \mod
\Gamma_2$, one can use inner automorphisms, such that $\phi$ acts like
the identity modulo $\Gamma_{d-1}$. 
\end{lemma} 
So now we have the following:
\begin{itemize}
\item Use the flip automorphism, if necessary, so that
  $\mathbf{N}_{i+1,i}^\phi=\mathbf{N}_{i+1,i}$ for $i=2,3,\ldots,d-2$.
\item Using extremal automorphism, if necessary, so that
  $\mathbf{N}_{i+1,i}^\phi=\mathbf{N}_{i+1,i}$ for $i=1,2,\ldots,d-1$.
\item Use a field automorphism and a diagonal automorphism, if necessary, so
  that $xe_{i+1,i}^\phi=xe_{i+1,i}\mod \Gamma_2$.
\item Use inner automorphisms, if necessary, so that
  $xe_{i+1,i}^\phi=xe_{i+1,i}\mod \Gamma_{d-1}$.
\item Use central automorphisms, if necessary, so that
  $xe_{i+1,i}^\phi=xe_{i+1,i}$. Note that the central automorphisms
  corresponding to $xe_{2,1}$ and $xe_{d,d-1}$ are inner automorphisms. 
\end{itemize}

Now we have proved the following theorem.
\begin{theorem}
The automorphism group of UT$(d,\mathbf{F})$ is generated by
extremal automorphisms, field automorphisms, diagonal
automorphisms, inner automorphisms and  central automorphisms.
\end{theorem}
\begin{acknowledgement}
Special thanks to I.B.S. Passi for his encouragement, reading the
whole manuscript and valuable comments.
\end{acknowledgement}
\nocite{*}
\bibliography{paper}
\bibliographystyle{amsplain}
\end{document}